\documentclass[10pt]{amsart}

\usepackage{amsmath, amsfonts, amsthm, amssymb, graphicx, fullpage, enumerate} 

\numberwithin{equation}{section}
 
\theoremstyle{plain}
\newtheorem{theorem}{Theorem}[section]
\newtheorem{lemma}[theorem]{Lemma}

\newtheorem{conjecture}[theorem]{Conjecture}

\theoremstyle{definition}

\newtheorem{case[theorem]}{Case}

\newtheorem*{claim*}{Claim}

\newenvironment{subproof}[1][\proofname]{%
  \begin{proof}[#1]%
}{%
  \end{proof}%
}

\theoremstyle{remark}

\numberwithin{equation}{section}

\def\R{ \mathbb{R}}

\def\d{\delta}

\newcommand{\Comment}[1]{}
\newcommand{\rbr}[1]{\left( {#1} \right)}

\newcommand{\cbr}[1]{\left\{ {#1} \right\}}

\def\RR{\mathbb{R}}

\def\NN{\mathbb{N}}
\def\OO{\mathbb{O}}
\def\supp{\text{supp}}
\def\diam{\text{diam}}
\def\stab{\text{Stab}}

\def\orb{\text{Orb}}

\def\d{\textnormal{\textbf{d}}}
\def\g{\textnormal{\textbf{g}}}
\def\h{\textnormal{\textbf{h}}}
\def\boldtheta{\boldsymbol{\theta}}
\def\m{\textnormal{\textbf{m}}}
\def\nug{\nu_{\g}}
\def\OOd{\prod_{i=1}^{\ell}\OO(d_i)}

\begin{document}

\title{Group actions and a multi-parameter Falconer distance problem} 


\author{Kyle Hambrook, Alex Iosevich, and Alex Rice}

\date{\today}

\address{Department of Mathematics, University of Rochester, Rochester, NY}
\email{khambroo@ur.rochester.edu, iosevich@math.rochester.edu, arice9@ur.rochester.edu} 
\thanks{This work was partially supported by NSA Grant H98230-15-1-0319 and NSERC}


\begin{abstract} In this paper we study the following multi-parameter variant of the celebrated Falconer distance problem (\cite{Falc86}). Given  $\d=(d_1,d_2, \dots, d_{\ell})\in \mathbb{N}^{\ell}$ with $d_1+d_2+\dots+d_{\ell}=d$ and $E \subseteq \R^d$, we define 
$$ 
\Delta_{\d}(E) = \left\{ \left(|x^{(1)}-y^{(1)}|,\ldots,|x^{(\ell)}-y^{(\ell)}|\right) : x,y \in E \right\} \subseteq \R^{\ell},
$$ where for $x\in \R^d$ we write $x=\left( x^{(1)},\dots, x^{(\ell)} \right)$ with $x^{(i)} \in \R^{d_i}$. 

We ask how large does the Hausdorff dimension of $E$ need to be to ensure that the $\ell$-dimensional Lebesgue measure of $\Delta_{\d}(E)$ is positive? We prove that if $2 \leq d_i$ for $1 \leq i \leq \ell$, then the conclusion holds provided
$$ \dim(E)>d-\frac{\min d_i}{2}+\frac{1}{3}.$$ 

\noindent We also note that, by previous constructions, the conclusion does not in general hold if 
$$\dim(E)<d-\frac{\min d_i}{2}.$$ 
A group action derivation of a suitable Mattila integral plays an important role in the argument. 
\end{abstract}

\maketitle

\setlength{\parskip}{5pt}

\section{Introduction}

Given a set $E \subseteq \RR^d$, the distance set of $E$ is
$$
\Delta(E) = \cbr{|x-y| : x,y \in E} \subseteq \RR.
$$
Falconer \cite{Falc86} studied how large the Hausdorff dimension of $E$ must be to guarantee that the Lebesgue measure of $\Delta(E)$ is positive. 
Falconer's conjecture is 
\begin{conjecture}\label{Falc Conj} 
Let $E$ be a compact subset of $\RR^d$, $d \geq 2$. If $\dim(E) > d/2$, then $|\Delta(E)| > 0$. 
\end{conjecture}
\noindent Here $| \; \cdot \; |$ is the Lebesgue measure and $\dim(\; \cdot \;)$ is the Hausdorff dimension. In \cite{Falc86}, Falconer showed that $d/2$ in the conjecture is best possible by constructing, for each $0 < s < d/2$, a compact set $E_s \subseteq \RR^d$ such that $\dim(E_s) = s$ and $\dim(\Delta(E_s)) \leq 2s/d$. Falconer's conjecture is open for all dimensions $d \geq 2$. Partial results have been obtained by Falconer \cite{Falc86}, Mattila \cite{Mat87}, Bourgain \cite{Bour94}, and others. The best currently known result, due to Wolff \cite{W99} ($d=2$) and Erdo\~{g}an \cite{Erd05} ($d \geq 3$), is
\begin{theorem}\label{Wolff-Erd}
Let $E$ be a compact subset of $\RR^d$, $d \geq 2$. If $\dim(E) > d/2 + 1/3$, then $|\Delta(E)| > 0$.
\end{theorem}

We will study a multi-parameter variant of Falconer's distance problem. Given $\d = (d_1,\dots,d_{\ell}) \in \mathbb{N}^{\ell}$, we let $d = d_1 + \cdots + d_{\ell}$. For $x \in \R^d$, we write 
$$
x = \left(x^{(1)},\ldots,x^{(\ell)}\right)
$$
where $x^{(i)} \in \R^{d_i}$. Given a set $E \subseteq \R^d$, we define the multi-parameter distance set of $E$ to be
$$
\Delta_{\d}(E) = \left\{ \left(|x^{(1)}-y^{(1)}|,\ldots,|x^{(\ell)}-y^{(\ell)}|\right) : x,y \in E \right\} \subseteq \R^{\ell}.
$$ 
Further, we let 
$$
\mathcal{F}(\d)=\sup \left\{ \dim(E) : E\subseteq \R^d, \ |\Delta_{\d}(E)|=0 \right\}. $$ 
\noindent By considering (a sequence of near) maximal dimensional sets with zero-measure distance sets in one hyperplane, crossed with full boxes in the other hyperplanes, we immediately have the relation 
\begin{align*}
\mathcal{F}(\d) \geq d - d_i+\mathcal{F}(d_i)
\end{align*} 
for all $1 \leq i \leq \ell$. 
Moreover, by the construction of Falconer \cite{Falc86} mentioned above, we have $\mathcal{F}(d_i) \geq d_i/2$ for all $1 \leq i \leq \ell$, 
and so 
%
\begin{align*}
\mathcal{F}(\d) \geq d - \frac{\min d_i}{2}. 
\end{align*}
 
Our main result is 

\begin{theorem}\label{main}  Let $\d=(d_1,\dots,d_{\ell}) \in \mathbb{N}^{\ell}$ with $2 \leq d_i$ for $1 \leq i \leq \ell$ and $d=d_1+\cdots+d_{\ell}$. 
If $E$ is a compact subset of $\RR^d$ with 
\begin{align}\label{dim hypothesis}
\dim(E)>d-\frac{\min d_{i}}{2}+\frac{1}{3}, 
\end{align}
then $\left|\Delta_{\d}(E)\right|>0$.

\end{theorem} 

\noindent In other words, Theorem \ref{main} is precisely the statement that 
$$
\mathcal{F}(\d)\leq d-\frac{\min d_{i}}{2}+\frac{1}{3}. 
$$  
%

Note that Theorem \ref{main} implies Theorem \ref{Wolff-Erd} by taking $\ell=1$. Note also that a similar problem has been studied in vector spaces over finite fields by Birklbauer and Iosevich \cite{PI}. 

The standard approach in studying Falconer's distance conjecture and related problems is to reduce the problem to the convergence of a so-called Mattila integral. This reduction is typically carried out via a stationary phase argument 
(see, for example, \cite{Bour94}, \cite{Erd05}, \cite{Mat87}, \cite{M15},  \cite{W99}, and references therein). 
Our approach is notable in that we instead carry out this reduction via the group action method developed by Greenleaf, Iosevich, Liu, and Palsson \cite{GILP15} in the study of the distribution of simplexes in compact sets of a given Hausdorff dimension. The method has its roots in the method developed by Elekes and Sharir in \cite{ES11}, which was ultimately used by Guth and Katz \cite{GK15} to prove the Erd\H os distance conjecture in the plane.

\section{Proof of Theorem \ref{main}} 

For the entirety of the proof, 
we fix $\d = (d_1,\dots,d_{\ell})\in \mathbb{N}^{\ell}$ with $2 \leq d_i$ for $1 \leq i \leq \ell$ and   
$
d=d_1+\cdots+d_{\ell}.
$    
We also fix a compact set $E \subseteq \RR^d$. 

The notation $A \lesssim B$ means there is a constant $C > 0$ such that $A \leq C B$; the constant may depend on $(d_1,\ldots,d_{\ell})$ and $E$, but not on any other parameters. Additionally, $A \gtrsim B$ means $B \lesssim A$, and $A \approx B$ means both $A \lesssim B$ and $B \lesssim A$. For $n \in \NN$, we let $\OO(n)$ denote the orthogonal group on $\RR^n$, and we note that $\OO(n)$ is a compact group with the operator norm topology.  

For each finite non-negative Borel measure $\mu$ supported on $E$, we define a measure $\nu$ on $\RR^{\ell}$ by 
$$
\int_{\RR^{\ell}} f(t) d\nu(t) = \int_{\RR^d} \int_{\RR^d} f(|x^{(1)}-y^{(1)}|,\ldots,|x^{(\ell)}-y^{(\ell)}|) d\mu(x) d\mu(y), 
$$
and, further, for each $\g = (g^{(1)},\ldots,g^{(\ell)}) \in \OOd$, we define a measure $\nug$ on $\RR^d$ by 
\begin{align*}
\int_{\RR^d} f(z) d\nug(z) = \int_{\RR^d} \int_{\RR^d} f(x^{(1)}-g^{(1)}y^{(1)},\ldots,x^{(\ell)}-g^{(\ell)}y^{(\ell)}) d\mu(x) d\mu(y).
\end{align*} 
We emphasize 
that $\nu$ and $\nu_\g$ both depend on $\mu$  
and that 
$
\text{supp}(\nu)\subseteq \Delta_{\d}(E).
$

Our goal is to show that, whenever \eqref{dim hypothesis} holds, 
there is a choice of $\mu$ for which the Fourier transform $\widehat{\nu}$ is in $L^2$. This will imply $\nu$ has an $L^2$ density with respect to Lebesgue measure on $\RR^{\ell}$, and hence $\left|\Delta_{\d}(E)\right|>0$. 

Our argument has two parts. In the first part, we exploit the action of the orthogonal group to show that, for any measure $\mu$ as above,  
$$
\int_{\RR^{\ell}} |\widehat{\nu}(\eta)|^2 d\eta 
\lesssim  
\int_{\RR^d} \int_{\OOd}  |\widehat{\nug}(\xi)|^2 d\g d\xi
\approx 
 \int_{\RR^d} |\widehat{\mu}(\xi)|^2 \int_{\prod_{i=1}^{\ell} S^{d_i - 1}}  |\widehat{\mu}(|\xi^{(1)}|\theta^{(1)},\ldots,|\xi^{(\ell)}|\theta^{(\ell)})|^2 d\boldtheta d\xi.
$$
This is split into Lemma \ref{nu to nug} and Lemma \ref{nug to mattila}.   
Here $d\g=dg^{(1)} \cdots dg^{(\ell)}$ is the product of the normalized Haar measures on $\OO(d_i)$, $i=1,\ldots,\ell$, and $d\boldtheta=d\theta^{(1)} \cdots d\theta^{(\ell)}$ is the product of the uniform probability measures on the spheres $S^{d_i - 1}$, $i=1,\ldots,\ell$.

In the second part of the argument, we use a slicing technique and a bound due to Wolff \cite{W99} ($n=2$) and Erdo\~{g}an \cite{Erd05} ($n \geq 3$) on the $L^2$ spherical average of the Fourier transform of a measure on $\RR^n$ to 
show that the multi-parameter Mattila integral  
$$
\int_{\RR^d}  |\widehat{\mu}(\xi)|^2 \int_{\prod_{i=1}^{\ell} S^{d_i - 1}}  |\widehat{\mu}(|\xi^{(1)}|\theta^{(1)},\ldots,|\xi^{(\ell)}|\theta^{(\ell)})|^2 d\boldtheta d\xi 
$$
is finite 
for some 
Frostman measure    
$\mu$ 
on $E$ whose existence is implied by the dimension hypothesis \eqref{dim hypothesis}. This is Lemma \ref{est of mattila}.

\subsection{Exploiting the Action of the Orthogonal Group}

\begin{lemma}\label{nu to nug} 
For any finite non-negative Borel measure $\mu$ supported on $E$,  
\begin{align*}
\int_{\RR^{\ell}} |\widehat{\nu}(\eta)|^2 d\eta \lesssim  \int_{\RR^d} \int_{\OOd}   |\widehat{\nug}(\xi)|^2 d\g d\xi.
\end{align*}
\end{lemma}
\begin{proof}
We begin by fixing approximate identities on $\RR^{\ell}$ and $\RR^{d}$ as follows. We choose $\phi \in C_c^{\infty}(\RR^{\ell})$ with $\phi \geq 0$, $\supp(\phi) \subseteq [-1,1]^{\ell}$, and $\int \phi(x) dx = 1$, and the associated approximate identity is $\phi_{\epsilon}(x)=\epsilon^{-\ell}\phi(\epsilon^{-1}x)$ for $\epsilon > 0$. Similarly, we choose $\psi \in C_c^{\infty}(\RR^{d})$ with $\psi \geq 0$, $\supp(\phi) \subseteq [-1,1]^{d}$, $\int \psi(x) dx = 1$, and $\psi \geq \frac{1}{2}$ on $[-\frac{1}{2},\frac{1}{2}]^d$, and the associated approximate identity is $\psi_{\epsilon}(x)=\epsilon^{-d}\psi(\epsilon^{-1}x)$ for $\epsilon > 0$. 

\noindent Since $\widehat{\phi_{\epsilon} \ast \nu} \to \widehat{\nu}$ and $\widehat{\psi_{\epsilon} \ast \nug} \to \widehat{\nug}$ uniformly as $\epsilon \to 0$, Plancherel's theorem tells us that Lemma \ref{nu to nug} will be proved upon establishing that, for all $\epsilon > 0$, 
\begin{align}\label{nu to nug approx 1}
\int_{\RR^{\ell}} (\phi_{\epsilon} \ast \nu)^2(t) dt \lesssim  \int_{\RR^d} \int_{\OOd} (\psi_{c\epsilon} \ast \nug)^2(z) d\g dz, 
\end{align}
where $c > 0$ is a constant 
depending only on the diameter of $E$.

\noindent For $\epsilon > 0$ and $\g \in \OOd$, we define the sets 
\begin{align*}
D(\epsilon) 
&= 
\cbr{(u,v,x,y) \in E^{4}: \left||x^{(i)}-y^{(i)}|-|u^{(i)}-v^{(i)}|\right| \leq \epsilon \quad \forall \ 1 \leq i \leq \ell }, 
\\
G(\epsilon,\g)
&= 
\cbr{ (u,v,x,y) \in E^{4}: |x^{(i)} - y^{(i)} - g^{(i)}(u^{(i)} -  v^{(i)})| \leq \epsilon \quad  \forall \  1 \leq i \leq \ell }.  
\end{align*}
We will establish \eqref{nu to nug approx 1} by proving the following three inequalities:
\begin{align}
\label{1 of 3}
\int_{\RR^{\ell}} (\phi_{\epsilon} \ast \nu)^2(t) dt
&\lesssim 
\epsilon^{-\ell} \mu^4(D({2\epsilon})),
\\
\label{2 of 3}
\epsilon^{-\ell} \mu^4(D(\epsilon)) 
&\lesssim 
\epsilon^{-d} \int_{\OOd} \mu^4(G(c \epsilon,\g)) d\g,
\\
\label{3 of 3}
\epsilon^{-d} \mu^4(G({\epsilon/4,\g})) 
&\lesssim 
\int_{\RR^d} (\psi_{\epsilon} \ast \nug)^2(z) dz,
\end{align}
where $\mu^4$ denotes the product measure $\mu \times \mu \times \mu \times \mu$, and $c = 2\max\cbr{2\text{diam}(E),1}$ in \eqref{2 of 3}.

\noindent 
We start by proving \eqref{1 of 3}. 

\noindent
For $t \in \RR^{\ell}$, we have 
\begin{align*}
\phi_{\epsilon} \ast \nu(t) 
&= \int_{\RR^d} \int_{\RR^d} \phi_{\epsilon}\left(t_1-|x^{(1)}-y^{(1)}|,\ldots,t_{\ell}-|x^{(\ell)}-y^{(\ell)}|\right) d\mu(x) d\mu(y) \\
&\lesssim \int_{\RR^d} \int_{\RR^d} \epsilon^{-\ell} \prod_{i=1}^{\ell} \chi\cbr{\left|t_i-|x^{(i)}-y^{(i)}|\right| \leq \epsilon} d\mu(x) d\mu(y),
\end{align*}
where $\chi A$ denotes the indicator function of a set $A$. 
Therefore, by the triangle inequality, 
\begin{gather*}
\int_{\RR^{\ell}} (\phi_{\epsilon} \ast \nu)^2(t) dt 
\\
\lesssim 
\epsilon^{-2\ell} \int \prod_{i=1}^{\ell} 
\chi\cbr{\left|t_i-|x^{(i)}-y^{(i)}|\right| \leq \epsilon} 
\chi\cbr{\left|t_i-|u^{(i)}-v^{(i)}|\right| \leq \epsilon} 
d\mu^4(u,v,x,y) dt 
\\
\leq 
\epsilon^{-2\ell} \int \prod_{i=1}^{\ell} 
\chi\cbr{\left|t_i-|x^{(i)}-y^{(i)}|\right| \leq \epsilon} 
\chi\cbr{\left||x^{(i)}-y^{(i)}|-|u^{(i)}-v^{(i)}|\right| \leq 2\epsilon}
d\mu^4(u,v,x,y) dt
\end{gather*}
For fixed $x^{(i)},y^{(i)} \in \RR^{d_i}$, the set of $t_i \in \RR$ with $\left|t_i-|x^{(i)}-y^{(i)}|\right| \leq \epsilon$ has Lebesgue measure $\approx \epsilon$. 
Thus integrating out $dt$ in the last integral yields \eqref{1 of 3}.

\noindent Now we prove \eqref{3 of 3}. 

\noindent 
Our choice of $\psi$ guarantees that $\psi_{\epsilon} \geq \frac{1}{2} \epsilon^{-d}$ on $[-\frac{1}{2}\epsilon,\frac{1}{2}\epsilon]^d$. 
Thus, for all $z \in \RR^d$, 
\begin{align*}
\psi_{\epsilon} \ast \nug(z) 
&= \int_{\RR^d} \int_{\RR^d} \psi_{\epsilon}(z^{(1)}-(x^{(1)}-g^{(1)}y^{(1)}),\ldots,z^{(\ell)}-(x^{(\ell)}-g^{(\ell)}y^{(\ell)})) d\mu(x) d\mu(y) \\
&\gtrsim \int_{\RR^d} \int_{\RR^d} \epsilon^{-d} \prod_{i=1}^{\ell} \chi\cbr{|z^{(i)}-(x^{(i)}-g^{(i)}y^{(i)})| \leq \frac{\epsilon}{2}} d\mu(x) d\mu(y).
\end{align*}
Therefore, by the triangle inequality, 
\begin{gather*}
\int_{\RR^d} (\psi_{\epsilon} \ast \nug)^2(z) dz \\
\gtrsim  
\epsilon^{-2d} \int \prod_{i=1}^{\ell} 
\chi\cbr{|z^{(i)}-(x^{(i)}-g^{(i)}u^{(i)})| \leq \frac{\epsilon}{2}} 
\chi\cbr{|z^{(i)}-(y^{(i)}-g^{(i)}v^{(i)})| \leq \frac{\epsilon}{2}} 
d\mu^4(u,v,x,y) dz
\\
\geq    
\epsilon^{-2d} \int   \prod_{i=1}^{\ell} 
\chi\cbr{|z^{(i)}-(x^{(i)}-g^{(i)}u^{(i)})| \leq \frac{\epsilon}{4}}
\chi\cbr{|(x^{(i)}-g^{(i)}u^{(i)})-(y^{(i)}-g^{(i)}v^{(i)})| \leq \frac{\epsilon}{4}} 
d\mu^4(u,v,x,y) dz
\\
=   
\epsilon^{-2d} \int   \prod_{i=1}^{\ell} 
\chi\cbr{|z^{(i)}-(x^{(i)}-g^{(i)}u^{(i)})| \leq \frac{\epsilon}{4}}
\chi\cbr{|x^{(i)}-y^{(i)}-g^{(i)}(u^{(i)}-v^{(i)})| \leq \frac{\epsilon}{4}} 
d\mu^4(u,v,x,y) dz.
\end{gather*}
For fixed $x^{(i)},u^{(i)} \in \RR^{d_i}$ and $g^{(i)} \in \OO(d_i)$, the set of $z^{(i)} \in \RR^{d_i}$ with $|z^{(i)}-(x^{(i)}-g^{(i)}u^{(i)})| \leq {\epsilon}/{4}$ has Lebesgue measure $\approx \epsilon^{d_i}$. 
Thus integrating out $dz$ in the last integral yields \eqref{3 of 3}.


\noindent Finally we prove \eqref{2 of 3}. 

\noindent 
Consider a fixed $1 \leq i \leq \ell$. 
For the action of $\OO(d_i)$ on $\RR^{d_i}$, the orbit of $e_{d_i}$ is $\orb(e_{d_i}) = \cbr{ge_{d_i} : g \in \OO(d_i)} = S^{d_i-1}$.   
We view the sphere $S^{d_i - 1}$ as a metric space with the Euclidean metric from $\RR^{d_i}$. 
We fix a cover of $S^{d_i - 1}$ by balls of radius $\epsilon$ such that the number of balls in the cover is $N(\epsilon,i) \approx \epsilon^{-(d_i - 1)}$  
and such that the cover has bounded overlap (that is, each set in the cover intersects no more than $C$ other sets in the cover, where $C$ is a constant independent of $\epsilon$). 
We let $T_{m_i}^{(i)}$ for $m_i=1,\ldots,N(\epsilon,i)$ denote the preimages of the balls with respect to the orbit map $g \mapsto ge_{d_i}$ from $\OO(d)$ to $S^{d_i - 1}$. 
Of course, the cover $\cbr{T_{m_i}^{(i)} : 1 \leq m_i \leq N(\epsilon,i)}$ of $\OO(d_i)$ also has bounded overlap. 
Moreover, since the image of the Haar measure on $\OO(d_i)$ with respect to the orbit map is exactly the uniform probability measure on $S^{d_i-1}$, each $T_{m_i}^{(i)}$ has measure $\approx \epsilon^{d_i - 1}$.

\noindent 
For each non-zero $w \in \RR^{d_i}$, we define the conjugation 
(change of basis) 
map $\zeta_{w}:\OO(d_i) \rightarrow \OO(d_i)$ by $\zeta_{w}(g) = pgp^{-1}$, where $p$ is a fixed but arbitrary transformation in $\OO(d_i)$ such that $p e_{d_i} = w/|w|$. 
For each $\epsilon > 0$ and $\g \in \OOd$, we define 
\begin{align*}
M(\epsilon) &= \cbr{ (m_1,\ldots,m_{\ell}) \in \NN^{\ell} : 1 \leq m_i \leq N(\epsilon,i) \quad \forall 1 \leq i \leq \ell }, 
\\
G'(\epsilon,\g) 
&= 
\left\{ 
(u,v,x,y) \in E^4 : |(x^{(i)} - y^{(i)}) - \zeta_{u^{(i)}-v^{(i)}}(g^{(i)}) (u^{(i)}-v^{(i)})| \leq \epsilon 
\quad \forall 1 \leq i \leq \ell
\right\}.
\end{align*}
\begin{claim*}
For any collection of transformations $g_{m_i}^{(i)} \in T_{m_i}^{(i)}$, $1 \leq i \leq \ell$, $1 \leq m_i \leq N(\epsilon,i)$, we have  
\begin{align*}
D(\epsilon) 
\subseteq 
\bigcup_{m \in M(\epsilon)}
G'(c\epsilon,\g_{m}),
\end{align*}
where $c=2\max\cbr{2\diam(E),1}$ and $\g_{m}=(g_{m_1}^{(1)},\ldots,g_{m_{\ell}}^{(\ell)})$. 
\end{claim*}
\begin{subproof}[Proof of Claim]
Let $u,v,x,y \in E$. It suffices to consider a fixed $1 \leq i \leq \ell$. 
Let $w = u^{(i)} - v^{(i)}$ and $z = x^{(i)} - y^{(i)}$. 
Assume $||z|-|w|| < \epsilon$. 
If $w=0$ or $z=0$, then $|z-gw| = ||z|-|w||< \epsilon$ for all $g \in \OO(d_i)$, and we are done. 
Assume $w$ and $z$ are non-zero. 
Choose $g \in \OO(d_i)$ such that $g(w/|w|) = z/|z|$, and hence $|z-gw|=||z|-|w|| < \epsilon$. 
Define $g_0 = \zeta_{w}^{-1}(g)$. 
We know $g_0 \in T_{m_i}^{(i)}$ for some $1 \leq m_i \leq N(\epsilon,i)$. 
Since $g_{m_i}^{(i)} \in T_{m_i}^{(i)}$ also, we have $|g_0e_{d_i} - g_{m_i}^{(i)}e_{d_i}| < 2\epsilon$. 
By the definition of $\zeta_{w}$, the previous inequality is equivalent to $|gw - \zeta_w(g_{m_i}^{(i)})w| < 2|w|\epsilon$. 
Therefore, by the triangle inequality,    
$|z - \zeta_w(g_{m_i}^{(i)})w| \leq \epsilon + 2|w|\epsilon \leq 2\max\cbr{2|w|,1}\epsilon$.  
To conclude, we note that $|w|=|u^{(i)}-v^{(i)}| \leq |u-v| \leq \diam(E)$.
\end{subproof}
\noindent 
For each $m \in M(\epsilon)$, we choose $\g_m = (g_{m_1}^{(1)},\ldots,g_{m_{\ell}}^{(\ell)}) \in \prod_{i=1}^{\ell} T_{m_i}^{(i)}$ such that 
\begin{align*}
\mu^4( G'(c\epsilon,\g_{\m}) )  \lesssim 
\epsilon^{-(d_1 - 1)} \cdots \epsilon^{-(d_{\ell} - 1)}
\int_{\prod_{i=1}^{\ell} T_{m_i}^{\epsilon,i}} 
\mu^4(G'(c\epsilon,\g)) \, d\g.
\end{align*}
Such a choice is possible because the average of a set must be larger than at least one element of the set.   
Then, using that 
$
\epsilon^{-\ell} \epsilon^{-(d_1 - 1)} \cdots \epsilon^{-(d_{\ell} - 1)} = \epsilon^{-d}
$, 
the claim implies 
\begin{align*}
\epsilon^{-\ell} \mu^4(D(\epsilon)) 
\leq 
\epsilon^{-\ell} 
\sum_{m \in M(\epsilon)} 
\mu^4(  G'(c\epsilon,\g_{\m})   )
\lesssim 
\epsilon^{-d} \sum_{\m \in M(\epsilon)} 
\int_{\prod_{i=1}^{\ell} T_{m_i}^{\epsilon,i}}
\mu^4(G'(c\epsilon,\g)) \, d\g.
\end{align*}
Expanding things out, the integral on the right equals 
\begin{gather*}
\int_{\prod_{i=1}^{\ell} T_{m_i}^{\epsilon,i}} 
\int_{E^{4}}
\chi \left\{ |(x^{(i)} - y^{(i)}) - \zeta_{u^{(i)}-v^{(i)}}(g^{(i)})  (u^{(i)}-v^{(i)})| \leq c\epsilon 
\quad \forall 1 \leq i \leq \ell
\right\} 
d\mu^4(u,v,x,y)
 d\g 
\\
= 
\int_{E^{4}}
\int_{\prod_{i=1}^{\ell} \zeta_{u^{(i)}-v^{(i)}} ( T_{m_i}^{\epsilon,i} )  } 
\chi \left\{ |(x^{(i)} - y^{(i)}) - g^{(i)} (u^{(i)}-v^{(i)}) | \leq c\epsilon 
\quad \forall 1 \leq i \leq \ell
\right\} 
 d\g 
d\mu^4(u,v,x,y).
\end{gather*} 
Thus, noting that $\cbr{\zeta_{w}(T_{m_i}^{(i)}) : 1 \leq m_i \leq N(\epsilon,i)}$ is a bounded overlap cover of $\OO(d_i)$ for each non-zero $w \in \RR^{d_i}$, we obtain 
\begin{align*}
\epsilon^{-\ell} \mu^4(D(\epsilon))
&\lesssim  
\epsilon^{-d} 
\int_{E^{4}}
\int_{\prod_{i=1}^{\ell} \OO(d_i) } 
\chi \left\{ |(x^{(i)} - y^{(i)}) - g^{(i)} (u^{(i)}-v^{(i)}) | \leq c\epsilon 
\quad \forall 1 \leq i \leq \ell
\right\} 
 d\g 
d\mu^4(u,v,x,y) 
\\
&= \epsilon^{-d} \int_{\prod_{i=1}^{\ell} \OO(d_i) } \mu^4( G(c\epsilon,\g)  ) d\g.
\end{align*}
\end{proof}


\begin{lemma}\label{nug to mattila}
For any finite non-negative Borel measure $\mu$ supported on $E$,  
\begin{align*}
\int_{\OOd} \int_{\RR^d}  |\widehat{\nug}(\xi)|^2 d\xi d\g \approx 
 \int_{\RR^d}  |\widehat{\mu}(\xi)|^2 \int_{\prod_{i=1}^{\ell}S^{d_i-1}} 
  |\widehat{\mu}(|\xi^{(1)}|\theta^{(1)},\ldots,|\xi^{(\ell)}|\theta^{(\ell)})|^2   
 d\boldtheta d\xi.
\end{align*}
\end{lemma}
\begin{proof}
By the definition of $\nug$, we have 
$$
\widehat{\nug}(\xi) = \widehat{\mu}(\xi) \widehat{\mu}(-(g^{(1)})^{T} \xi^{(1)},\ldots,-(g^{(\ell)})^{T} \xi^{(\ell)}),
$$
where $T$ indicates transpose. Therefore 
\begin{align*}
\int_{\RR^d} \int_{\OOd} |\widehat{\nug}(\xi)|^2 d\g d\xi 
= 
\int_{\RR^d} |\widehat{\mu}(\xi)|^2 \int_{\OOd} |\widehat{\mu}(g^{(1)} \xi^{(1)},\ldots,g^{(\ell)} \xi^{(\ell)})|^2 d\g d\xi.
\end{align*}
We now consider the inner integral on the right for fixed non-zero $\xi \in \RR^d$. 
By a change of variable and the translation invariance of the Haar measures, 
\begin{align*}
\int_{\OOd} |\widehat{\mu}(g^{(1)} \xi^{(1)},\ldots,g^{(\ell)} \xi^{(\ell)})|^2 d\g 
=
\int_{\OOd} 
|\widehat{\mu}(g^{(1)} e_{d_1} |\xi^{(1)}|,\ldots,g^{(\ell)} e_{d_{\ell}}|\xi^{(\ell)})|^2 
d\g, 
\end{align*}
where $e_{d_i} = (0,\ldots,0,1) \in \RR^{d_i}$. The stabilizer subgroup of $\OO(d_i)$ for $e_{d_i}$ is $\stab(e_{d_i}) = \cbr{g \in \OO(d_i) : ge_{d_i} = e_{d_i}}$. 
As $\OO(d_i)$ is compact and $\stab(e_{d_i})$ is closed, $\stab(e_{d_i})$ is compact. 
We equip $\stab(e_{d_i})$ with its normalized Haar measure. The quotient space $\OO(d_i)/\stab(e_{d_i})$ is homeomorphic to the sphere $S^{d_i - 1}$. The measure on $\OO(d_i)/\stab(e_{d_i})$ is the image of the uniform probability measure on $S^{d_i - 1}$; it is a left-invariant Radon measure.  
Putting all this together, by the quotient integral formula (see, for example, \cite{DE08}, \cite{Folland94}), the last integral above equals a constant multiple of 
\begin{gather*}
\int_{\prod_{i=1}^{\ell} \OO(d_i)/\stab(e_{d_i})} \int_{\prod_{i=1}^{\ell} \stab(e_{d_i}) } |\widehat{\mu}(g^{(1)} h^{(1)} e_{d_1} |\xi^{(1)}|,\ldots,g^{(\ell)} h^{(\ell)} e_{d_{\ell}}|\xi^{(\ell)}|)|^2 d\h d\g
\\
= 
\int_{\prod_{i=1}^{\ell} \OO(d_i)/\stab(e_{d_i})} \int_{\prod_{i=1}^{\ell} \stab(e_{d_i}) } |\widehat{\mu}(g^{(1)}  e_{d_1} |\xi^{(1)}|,\ldots,g^{(\ell)}  e_{d_{\ell}}|\xi^{(\ell)}|)|^2 d\h d\g
\\
= 
\int_{\prod_{i=1}^{\ell} \OO(d_i)/\stab(e_{d_i})}  |\widehat{\mu}(g^{(1)}  e_{d_1} |\xi^{(1)}|,\ldots,g^{(\ell)}  e_{d_{\ell}}|\xi^{(\ell)}|)|^2 d\g
\\
=
\int_{\prod_{i=1}^{\ell} S^{d_i - 1}}  |\widehat{\mu}( |\xi^{(1)}| \theta^{(1)} ,\ldots, |\xi^{(\ell)}| \theta^{(\ell)} )|^2  d\boldtheta.
\end{gather*}



\end{proof}

\subsection{Estimating the Multi-Parameter Mattila Integral}

\begin{lemma}\label{est of mattila} If $$\dim(E) > d-\frac{\min d_i}{2}+\frac{1}{3},$$ then there exists a finite non-negative Borel measure $\mu$ supported on $E$ satisfying 
\begin{align}\label{mattila integral}
 \int_{\R^d}  |\widehat{\mu}(\xi)|^2  \int_{\prod_{i=1}^{\ell}S^{d_i-1}} |\widehat{\mu}( |\xi^{(1)}|\theta^{(1)},\ldots,|\xi^{(\ell)}|\theta^{(\ell)} )|^2   d\boldtheta d\xi < \infty. 
\end{align} 
\end{lemma}

\noindent 
For the proof of Lemma \ref{est of mattila}, we need two lemmas. 
The first is an estimate for the $L^2$ spherical average of the Fourier transform due to Wolff \cite{W99} ($n=2$) and Erdo\~{g}an \cite{Erd05} ($n \geq 3$).
\begin{lemma}\label{wolffest} Let $\lambda$ be a finite compactly supported Borel measure on $\R^n$. If $t,\epsilon>0$ and 
$$
\frac{n}{2}\leq \alpha \leq \frac{n+2}{2},
$$ then 
\begin{equation*}
\int_{S^{n-1}} |\widehat{\lambda}(t \theta )|^2 d \theta \leq C_{\epsilon} t^{ -\frac{ n+2\alpha-2}{ 4}+\epsilon} I_{\alpha}(\lambda), 
\end{equation*}
where 
\begin{align}\label{energy integral}
I_{\alpha}(\lambda)
=
\int_{\R^{n}}\int_{\R^{n}}|x-y|^{-\alpha} d\lambda(x)d\lambda(y) = C_{n,\alpha} \int_{\R^n} |\widehat{\lambda}(\xi)|^2 |\xi|^{-n+\alpha} d\xi.
\end{align}
\end{lemma}

\noindent 
The second lemma that we need 
examines the behavior of Frostman-type measures on Cartesian products of differently-sized balls from coordinate hyperplanes. Here we let $B^d_{\delta}(x)$ denote the ball in $\R^d$ of radius $\delta$ centered at $x$, while we let $R^d_{\delta}(x)$ denote the box $x+[-\delta,\delta]^d \subseteq \R^d$.

\begin{lemma}\label{frost}  Suppose $0< s \leq d$ and $\mu$ is a finite Borel measure on $\RR^d$ satisfying 
$$
\mu\left(B^d_{\delta}(x)\right) \lesssim \delta^s
$$ 
for all $x\in \R^d$ and $\delta > 0$. 
If $x=(x^{(1)},\dots,x^{(\ell)})\in \R^d$ and $\delta_1, \dots, \delta_{\ell} > 0$ 
with $\delta_j \leq \delta_i$ for all $1\leq i \leq \ell$, 
then 
$$
\mu\left( B^{d_1}_{\delta_1}(x^{(1)})\times \cdots \times  B^{d_{\ell}}_{\delta_{\ell}}(x^{(\ell)}) \right) \lesssim \delta_j^{s-(d-d_j)}\prod_{i\neq j}\delta_i^{d_i}. 
$$ 
In particular, if $x=(x^{(1)},\dots,x^{(\ell)})\in \R^d$ and $0 < \delta_1,\dots, \delta_{\ell} \leq 1$, then   
$$
\mu\left( B^{d_1}_{\delta_1}(x^{(1)})\times \cdots \times  B^{d_{\ell}}_{\delta_{\ell}}(x^{(\ell)}) \right) \lesssim \prod_{i=1}^{\ell} \delta_i^{s-(d-d_i)}. 
$$
\end{lemma}

\begin{proof} For technical ease, we proceed using boxes instead of balls, noting that the results are equivalent. Fixing $x\in \R^d$ and $\delta_1,\dots, \delta_{\ell} > 0$ with  $\delta_j\leq \delta_i$ for all $1\leq i \leq \ell$, we see that 
$R=R^{d_1}_{\delta_1}(x^{(1)})\times \cdots \times  R^{d_{\ell}}_{\delta_{\ell}}(x^{(\ell)})$ 
is precisely obtained by stretching $R^d_{\delta_j}(x)$ by a factor of $\delta_i/\delta_j$ in each hyperplane, so in particular $R$ is contained in $\prod_{i\neq j} \lceil \delta_i/\delta_1 \rceil^{d_i}$ translated copies of $R^d_{\delta_j}(x)$. 

\noindent Therefore, 
$$\mu(R) \lesssim \delta_j^s\prod_{i\neq j} \lceil \delta_i/\delta_j \rceil^{d_i},  $$ 
and the lemma follows.
\end{proof}

\begin{proof}[Proof of Lemma \ref{est of mattila}] 
Given the hypotheses of the lemma, we let $s=\dim(E)$, we define $\epsilon>0$ by $$4\epsilon = s-\left(d-\frac{\min d_i}{2}+\frac{1}{3}\right), $$ and we let $\mu$ be any finite non-negative Borel measure supported on $E$ satisfying 
\begin{equation}\label{ballcond}
\mu\left(B^d_{\delta}(x)\right) \lesssim \delta^{s-\epsilon}  
\end{equation} 
for all $x \in \RR^d$ and $\delta > 0.$ 
The existence of $\mu$ is guaranteed by Frostman's Lemma (see, for example, \cite{Falc86II}, \cite{M15}). 
We will estimate the integral in \eqref{mattila integral} by iteratively applying Lemma \ref{wolffest} to ``Fourier slice" measures $\lambda_i$ on $\R^{d_i}$, defined for fixed \  $\xi^{(1)},\dots,\xi^{(i-1)},$ $\xi^{(i+1)},\dots ,\xi^{(\ell)}$ by $$\widehat{\lambda_i}(\xi^{(i)})=\widehat{\mu}\left(\xi^{(1)},\dots,\xi^{(\ell)}\right).$$
Indeed, the integral in \eqref{mattila integral} is
\begin{gather*} 
\int_{\R^d} |\widehat{\mu}(\xi)|^2 \int_{\prod_{i=2}^{\ell}S^{d_i-1}}
\left(\int_{S^{d_1-1}} 
|\widehat{\mu}( |\xi^{(1)}|\theta^{(1)}, \dots ,|\xi^{(\ell)}|\theta^{(\ell)} ) |^2   d\theta^{(1)} \right)  d\theta^{(2)}\cdots d\theta^{(\ell)} d\xi
\\
\lesssim  
\int_{\R^d}  |\widehat{\mu}(\xi)|^2 \int_{\prod_{i=2}^{\ell}S^{d_i-1}} 
|\xi^{(1)}|^{-\frac{d_1+2\alpha_1-2}{4}+\epsilon} 
\left(
\int_{\R^{d_1}} |\widehat{\mu}( \eta^{(1)}, \dots ,|\xi^{(\ell)}|\theta^{(\ell)} )|^2 
|\eta^{(1)}|^{-d_1+\alpha_1}  d \eta^{(1)}
\right) 
d\theta^{(2)}\cdots d\theta^{(\ell)} d\xi  
\\ \vdots \\
\lesssim  
\left( \int_{\R^d} |\widehat{\mu}(\xi)|^2 \prod_{i=1}^{\ell} |\xi^{(i)} |^{-\frac{d_i+2\alpha_i-2}{4}+\epsilon} d \xi\right) \left( \int_{\R^d} |\widehat{\mu}(\eta)|^2  \prod_{i=1}^{\ell} |\eta^{(i)} |^{-d_i+\alpha_i} d \eta \right),
\end{gather*} 
\noindent provided that $\frac{d_i}{2}\leq \alpha_i \leq \frac{d_i+2}{2}$ for all $1\leq i \leq \ell$. 
Expressing the integrals from the last line 
on the space side (using the ``Fourier slice'' measures and  \eqref{energy integral}),  
we obtain a constant multiple of 
\begin{align*}
\rbr{
\int_{\R^d}\int_{\R^d} \prod_{i=1}^{\ell}|x^{(i)}-y^{(i)}|^{\frac{d_i+2\alpha_i-2}{4}-\epsilon-d_i} d \mu(x) d \mu(y)
}
\rbr{
\int_{\R^d}\int_{\R^d} \prod_{i=1}^{\ell} |x^{(i)}-y^{(i)}|^{-\alpha_i} d \mu(x) d \mu(y)
}.
\end{align*}

\noindent By decomposing dyadically into regions where 
$2^{-j_i}\leq |x^{(i)}-y^{(i)}| \leq 2^{-j_i+1}$ 
and then applying Lemma \ref{frost} and \eqref{ballcond}, we see that 
convergence of 
the integrals 
is implied by convergence of the sums 
\begin{align*}
\prod_{i=1}^{\ell} \sum_{j_i=0}^{\infty} 2^{-j_i\left(\frac{d_i+2\alpha_i-2}{4}-d_i+s-(d-d_i)-2\epsilon\right)} 
\qquad \text{ and } \qquad
\prod_{i=1}^{\ell}\sum_{j_i=0}^{\infty}2^{-j_i\left(-\alpha_i+s-(d-d_i)-\epsilon\right)}.
\end{align*}

\noindent 
Convergence of 
the former sum is equivalent to 
$\frac{d_i+2\alpha_i-2}{4}+s-(d-d_i)>d_i+2\epsilon $ 
for all $1\leq i \leq \ell$.
Convergence of 
the latter sum is equivalent to $\alpha_i< s-(d-d_i)-\epsilon$ for all $1\leq i \leq \ell$. 
Recalling the definition of $\epsilon$ and the requirement  that $\frac{d_i}{2}\leq \alpha_i\leq \frac{d_i+2}{2}$ 
for all 
$1\leq i \leq \ell$,  
we see that all inequalities can be satisfied by setting $$\alpha_i= \min \left\{s-(d-d_i)-2\epsilon,\frac{d_i+2}{2} \right\}.$$
\end{proof}

\vskip.125in 

\end{document}